\documentclass[13pt]{amsart}

\usepackage{multirow}
\usepackage{latexsym}
\usepackage{amsmath, amsfonts, amssymb, subeqnarray}
\usepackage{graphicx,epsf}
\usepackage[colorlinks]{hyperref}
\usepackage{enumerate}
\usepackage[T1]{fontenc}
\usepackage{symbols}
\usepackage{enumitem}
\newcommand{\bm}[1]{\mbox{\boldmath$#1$}}

\newtheorem{Theorem}{Theorem}[section]
\newtheorem{Lemma}{Lemma}[section]

\newtheorem{Corollary}{Corollary}[section]

\newtheorem{definition}{Definition}[section]

\newtheorem{remark}{Remark}[section]

\numberwithin{equation}{section} \numberwithin{table}{section}
\numberwithin{figure}{section}

\newcommand{\Label}[1]{\label{#1}\mbox{\fbox{\bf #1}\quad}}
\renewcommand{\Label}[1]{\label{#1}}

\newcommand{\Null}[1]{\mathcal{N}(#1)}
\renewcommand{\ker}[1]{\Null}

\everymath{\displaystyle}

\begin{document}

\title{A Sharp Korn's Inequality for Piecewise $H^1$ Space and its applications} 
\date{\today}

\begin{abstract}
In this paper, we revisit Korn's inequality for the piecewise $H^1$ space based on general polygonal or polyhedral decompositions of the domain. Our Korn's inequality is expressed with minimal jump terms. These minimal jump terms are identified by characterizing the restriction of rigid body mode to edge/face of the partitions. Such minimal jump conditions are shown to be sharp for achieving the Korn's inequality as well. The sharpness of our result and explicitly given minimal conditions can be used to test whether any given finite element spaces satisfy Korn's inequality, immediately as well as to build or modify nonconforming finite elements for Korn's inequality to hold. 
\end{abstract}

\author{Qingguo Hong$^\dag$,  Young-Ju Lee$^\ddag$ and Jinchao Xu}
\address{Department of Mathematics, Penn State University, PA, USA} 
\address{Department of Mathematics, Texas State University, TX, USA} \email{huq11@psu.edu, yjlee@txstate.edu, xu@math.psu.edu}
\thanks{\small The second author is partially supported by Shapiro Fellowship from Penn State in Spring of 2022, and he is also supported by NSF DMS-2208499}
\maketitle

\section{Introduction}\Label{Intro}
The Korn's inequality played a fundamental role in the development of linear elasticity. There is a work reviewing Korn's inequality and its applications in continuum mechanics \cite{horgan1995korn}. In fact, there are a lot of works on proving classical Korn's inequality \cite{nitsche1981korn,wang2003korn,ciarlet2010korn} for functions in $H^1$ vector fields. In \cite{brenner2004korn}, Korn’s inequalities for piecewise $H^1$ vector fields are established. We note that a strengthened version of Korn’s inequality for piecewise $H^1$ vector fields is presented in \cite{mardal2006observation}. By the Korn's inequality established in 
\cite{mardal2006observation}, the nonconforming finite element for 2D introduced in \cite{mardal2002robust} is shown to satisfy the Korn's inequality. In this paper, we revisit Korn's inequality for the piecewise $H^1$ space based on general polygonal or polyhedral decompositions of the domain. Our Korn's inequality is expressed with minimal jump terms, which can readily be used to design degree of freedoms. Namely, we move one step further from the result presented in \cite{mardal2006observation}. In particular, for 3D, we identify that the tangential component of rigid body mode restricted to the face of triangle is indeed the lowest Raviart-Thomas element on face, which is the rotated version discussed in \cite{mardal2006observation}. Additionally, with the minimal jump terms, we show that our Korn's inequality is sharp. We emphasize that with the help of sharpness of the Korn's inequality and explicit form of jump conditions, it is very easy to test whether any given finite element spaces satisfy the Korn's inequality. The proposed minimal jump conditions can also provide a guidance to modify some existing nonconforming finite elements so that the Korn's inequality is satisfied.

Throughout the paper, we shall use the standard notation for Sobolev spaces. Namely, $H^k(\Omega)$ denotes the Sobolev space of scalar functions in $\Omega$ whose derivatives up to order $k$ are square integrable, with the norm $\|\cdot\|_k$. The notation $|\cdot |_k$ denotes the semi-norm derived from the partial derivatives of order equal to $k$. Furthermore, $\|\cdot\|_{k,T}$ and $|\cdot|_{k,T}$ denote the norm $\|\cdot\|_k$ and the semi-norm $|\cdot|_k$ restricted to the domain $T$. Given a partitions $\mathcal{T}_h$ of the domain $\Omega$, where $\Omega$ is a bounded connected open polyhedral domain in $\Reals{d}$ with $d = 2$ or $3$. , we shall also use $H^1(\Omega;\mathcal{T}_h)$ to denote the element-wise $H^1$ functions. We denote the vectors of size $d$ whose components are in $H^1(\Omega;\mathcal{T}_h)$ by $(H^1(\Omega;\mathcal{T}_h))^d$. We also denote $H({\rm div};\Omega)$ by the space consisting of vectors, whose divergence belongs to $L^2(\Omega)$. We use $P_\ell(T)$ for the space of polynomials of degree upto $\ell$ on the domain $T$ while $(P_\ell(T))^d$ denotes the vectors of size $d$ whose components are polynomials of degree at most $\ell$. 

We recall that the operator ${\bm{curl}}$ on a scalar function $q$ in $2D$ is defined by
\begin{equation}
{\bm{curl}} q = \left ( -\frac{\partial q}{\partial y}, \frac{\partial q}{\partial x} \right )^T, 
\end{equation}
while on a vector function $\bm{q} = (q_i)_{i=1,2,3}^T$ in $3D$, it is defined by 
\begin{equation}
{\bm{curl}} \bm{q} = \left ( \frac{\partial q_2}{\partial z} - \frac{\partial q_3}{\partial y}, \frac{\partial q_3}{\partial x} - \frac{\partial q_1}{\partial z}, \frac{\partial q_1}{\partial y} - \frac{\partial q_2}{\partial x} \right )^T, 
\end{equation}
For any given vector space $\bm{V}$, by ${\rm dim} \bm{V}$, we mean the dimension of $\bm{V}$. 

The rest of our paper is organized as follows. In \S \ref{korncd}, we prove the Korn's inequality for piecewise $H^1$ space with minimal jump terms. \S \ref{sharp} considers the sharpness of the inequality, which indicates that the minimal jump terms presented in our Korn's inequality are necessary and can not be reduced any more. With the  Korn's inequality for piecewise $H^1$ space with minimal jump terms and sharpness discussion, in \S \ref{appl}, some of existing finite elements have been discussed. And some application of our theory is provided to modify nonconforming finite elements which do not satisfy the Korn's inequality so that they satisfy the Korn's inequality.  
Lastly, we provide a concluding remark. 

\section{Korn's inequality for the piecewise $H^1$ vector functions}\label{korncd} 

Let $\Omega$ be a bounded connected open polyhedral domain in $\Reals{d}$ with $d = 2$ or $3$ and $\partial \Omega$ be the boundary of the domain $\Omega$. Then the classical Korn's inequality reads as follows:
\begin{equation}
|{\bf{u}} |_{H^1(\Omega)} \leq C_\Omega \left ( \| \mathcal{D}({\bf{u}})\|_0 + \|{\bf{u}}\|_0 \right ), \quad \forall {\bf{u}} \in [H^1(\Omega)]^d, 
\end{equation} 
where the strain tensor $\mathcal{D}({\bf{u}}) \in \Reals{d\times d}$ is given as follows: 
\begin{equation}
\mathcal{D}_{ij}({\bf{u}}) = \frac{1}{2} \left ( \frac{\partial u_i}{\partial x_j} + \frac{\partial u_j}{\partial x_i} \right ) \quad \mbox{ for } 1 \leq i,j \leq d.  
\end{equation} 
Let ${\bm{RM}}(\Omega)$ be the space of rigid motions on $\Omega$ defined by 
\begin{equation}
\bm{RM} (\Omega) = \{ \bm {a} + \bm{A} \bm{x} \in \Reals{d} : \bm{a} \in \Reals{d} \,\, \mbox{ and } \,\, \bm{A} \in {\textsf{Skw}}^{d\times d} \},  
\end{equation}
where $\bm{x} = (x_1, \cdots, x_d)^T \in \Reals{d}$ is the position vector function  on $\Omega$ and ${\textsf{Skw}^{d\times d}}$ is the set of anti-symmetric $d\times d$ matrices. We would like to remark that if $\bm{v} = \bm{v}(\bm{x}) \in \bm{RM}(\Omega)$, then $\bm{v} = \bm{a} + \bm{A}\bm{x}$ for some constant vector $\bm{a}$ and a homogeneous polynomial of degree one  $\bm{A}\bm{x}$ such that $\bm{x} \cdot \bm{A} \bm{x} = 0$. This is relevant to the lowest N\'ed\'elec first kind finite element \cite{monk2003finite}. It is easy to verify that the space $\bm{RM}(\Omega)$ is the kernel of the strain tensor, i.e., it holds that  
\begin{equation}
\bm{RM}(\Omega) = \{\bm{v} \in (H^1(\Omega))^d : \mathcal{D}(\bm{v}) = 0 \}. 
\end{equation} 

\subsection{Korn's inequality with explicit upper bounds}
Given $\Omega \subset \Reals{d}$ with $d = 2$ or $3$, we consider the mesh generation $\mathcal{T}_h$, which is a shape-regular polygonal and/or polyhedral partition for $\Omega$. Let $\mathcal{T}_h = \cup \{T\}$ denote the collection of decompositions and $h = \max_{T \in \mathcal{T}_h} {\rm diam} (T)$ denote the mesh size. We further denote by $\mathcal{E}_h$ the set of all edges/faces for $\mathcal{T}_h$ and 
\begin{equation} 
\mathcal{E}_h = \mathcal{E}_h^{o} \cup \mathcal{E}_h^\partial, 
\end{equation}
where $\mathcal{E}_h^o$ denotes the set of all interior edges/faces of $\mathcal{T}_h$ and $\mathcal{E}_h^{\partial}$ denotes the set of all boundary edges/faces, respectively. Let $\mathcal{V}(T)$ be the set of all vertices for $T \in \mathcal{T}_h$ while $\mathcal{V} (\mathcal{T}_h)$ denotes the set of all vertices for the 
partition $\mathcal{T}_h$.  

Given two adjacent polygon/polyhedron $T^+$ and $T^-$ in  $\mathcal{T}_h$, let $f = \partial T^+ \cap \partial T^-$ 
be the common boundary (interface) between $T^+$ and $T^-$ in $\mathcal{T}_h$, and $n^+$ and $n^-$ be unit normal vectors to $f$ pointing to the exterior of $T^+$ and $T^-$, respectively. For any edge (or face) $f
\in \mathcal{E}_h^{o}$, and a scalar $q$ and vector $\bm{v}$, we define the jumps
\begin{subeqnarray}
\jump{q}_f &=& q|_{\partial T^+\cap f} - q|_{\partial T^-\cap f}, \\ 
\jump{\bm{v}}_f &=& \bm{v}|_{\partial T^+\cap f} - \bm{v}|_{\partial T^-\cap f}.
\end{subeqnarray}
When $f \in \mathcal{E}_h^{\partial}$ then the above quantities are defined as
\begin{equation} 
\jump{q}_f = q|_{f}, \quad \mbox{ and }  \quad \jump{\bm{v}} = \bm{v}|_{f}. 
\end{equation} 
Throughout the paper, we shall also consider the subspace of $\bm {RM}(\Omega)$, denoted by $\bm {RM}^\partial (\Omega)$ defined as follows: 
\begin{equation} 
\bm {RM}^{\partial}(\Omega) = \left \{\bm m \in \bm {RM}(\Omega): ~ \|\bm m\|_{L^2(\partial \Omega)}=1,\int_{\partial \Omega}\bm mds=0 \right \}. 
\end{equation} 
We also consider the following space of piecewise linear vector fields: 
\begin{eqnarray*} 
\bm{V}_h &=& \{\bm v\in (L^2(\Omega))^d:  \bm v_T = \bm v|_T \in (P_1(T))^d, \quad \forall T \in \mathcal {T}_h\} 
\end{eqnarray*}
and space of continuous piecewise linear vector fields:
\begin{eqnarray*} 
\bm{W}_h &=& \{\bm v\in (H^1(\Omega))^d:   \bm v_T =\bm v|_T \in (P_1(T))^d, \quad \forall T \in \mathcal{T}_h\}. 
\end{eqnarray*}
Consider a linear map $E: \bm{V}_h \longrightarrow \bm{W}_h$ defined as follows: 
\begin{equation} 
E (\bm v)(p) = \frac{1}{|\chi_p|}\sum_{T \in \chi_p}\bm v|_T(p), \quad \forall p \in \mathcal{V}(\mathcal{T}_h),
\end{equation} 
where $\chi_p = \{ T \in \mathcal{T}_h: p \in \mathcal{V}(T) \}$, the patch of the vertex $p$, i.e., the collection simplexes that contain $p$ as its vertex, and $|\chi_p|$ is the cardinality of the set $\chi_p$. We note that it holds true 
\begin{equation} 
|\chi_p| \lesssim 1, \quad \forall p \in \mathcal{V}(\mathcal{T}_h).
\end{equation} 
We recall the following approximation property, which can be found in \cite{brenner2004korn}:
\begin{Lemma}
Let $\bm v \in \bm{V}_h$ and $T \in \mathcal{T}_h$. Then, with $\bm v_{T} = \bm v|_T$,  for all $p \in \mathcal{V}(\mathcal{T}_h)$, we have the following estimate: 
\begin{equation} 
|\bm v_{T} - E(\bm v)(p)|^2 \lesssim \sum_{f \in \mathcal{E}_p} |\jump{\bm v}_f(p)|^2, \quad T \in \mathcal{T}_h \, \mbox{ and } \, p \in \mathcal{V}(T),
\end{equation} 
where 
\begin{equation} 
\mathcal{E}_p = \{ e \in \mathcal{E}_h^o : p \in \partial f \}, 
\end{equation} 
is the set of interior sides sharing $p$ as a common vertex, and $\jump{\bm v}_f$ is the jump of $\bm v$ across the edge/face $f$.
\end{Lemma}

The main observation that will lead us to construct the minimal jump conditions or refined Korn's inequality will be presented at the following simple but important theorem: 
\begin{Theorem}\label{kornpre} 
For any $T \in \mathcal{T}_h$. Let $f \in \partial T$ and $\bm{c}_{f}$ be the barycenter of $f$. For some constant $c, c_1, c_2 \in \Reals{}$, we have for 2D and 3D, respectively, with $\bm{t}_f$ and $\bm{n}_f$ being the tangent vector to the edge and the normal vector to the face,  
\begin{subeqnarray*}
{\bm{RM}}(f)^\perp &=& \{ (\bm{v} \cdot \bm{t}_f) |_f: \forall \bm{v} \in \bm{RM}(T) \} = {\rm span} \{ 1\} \quad \mbox{ for } d = 2 \\
&=& {\rm RT}_0(f) := \{ (\bm{v} \times \bm{n}_f)|_f : \forall \bm{v} \in \bm{RM}(T) \} \quad \mbox{ for } d = 3, 
\end{subeqnarray*} 
where ${\rm RT}_0(f)$ can be characterized as follows: 
\begin{equation}
{\rm RT}_0(f) = \{ (c_1 + c x_1) \bm{t}_1 + (c_2 + c x_2) \bm{t}_2 : x_1 , x_2 \in \Reals{} \}, 
\end{equation} 
where $f = {\rm span} \{\bm{t}_1, \bm{t}_2 \}$ such that $\bm{t}_1 \cdot \bm{t}_2 = 0$. 
%
%
\end{Theorem}
\begin{proof}
We begin our proof for the {\rm 2D} case. Choose $\bm{v} \in {\bm{RM}}(T)$. Then, $\bm{v}$ is of the following form: 
\begin{equation}
\bm{v} = \bm{a} + b (y, -x)^T, 
\end{equation} 
for some constant vector $\bm{a}$ and a constant $b$. Without loss of generality, we assume that $f \in \partial T$ can be expressed as a linear function $y = mx + n$. Then, we have that $\bm{t}_f =\frac{1}{\sqrt{1+m^2}} (1,m)^T$ and 
\begin{equation} 
(\bm{v} \cdot \bm{t}_f) |_f= c_0, 
\end{equation}
where $c_0$ is some constant. We shall now consider {\rm 3D} case. For any $\bm{v} \in {\bm{RM}}(T)$ can be given as $\bm{v} = \bm{a} + \bm{A} \bm{x}$, where $\bm{A}$,  $\bm{x}$, and $\bm{n}_f$ can be denoted by the followings: 
\begin{eqnarray*}
\bm{A} = \left ( \begin{array}{ccc} 0 & a_1 & a_2 \\ -a_1 & 0 & a_3 \\ -a_2 & -a_3 & 0  \end{array} \right ), \quad  \bm{x} = \left ( \begin{array}{c} x_1 \\ x_2 \\ x_3 \end{array} \right ) 
, \quad \mbox{ and } \quad \bm{n}_{f} = \left ( \begin{array}{c} n_1 \\ n_2 \\ n_3 \end{array} \right ). 
\end{eqnarray*}
Then a simple calculation leads to 
\begin{eqnarray*}
\bm{v} \times \bm{n}_{f} = \bm{a} \times \bm{n}_f + \left ( \begin{array}{c} a_1 x_2 + a_2 x_3 \\ -a_1 x_1 + a_3x_3  \\ -a_2x_1 - a_3x_2 \end{array} \right ) \times  \left ( \begin{array}{c} n_1 \\ n_2 \\ n_3 \end{array} \right ).
\end{eqnarray*}
Particularly, we pay attention to the following quantity: 
\begin{eqnarray*}
\bm{A} \bm{x} \times \bm{n}_{f} = \left ( \begin{array}{c} \eta x_1 + a_3 \bm{n}_f \cdot \bm{x} \\ \eta x_2 - a_2 \bm{n}_f \cdot \bm{x}  \\ \eta x_3 + a_1 \bm{n}_f \cdot \bm{x} \end{array} \right ) = 
\left ( \begin{array}{c} \eta x_1 + a_3 \bm{n}_f \cdot (\bm{x} - \bm{c}_f) + a_3 \bm{n}_f \cdot \bm{c}_f \\ \eta x_2 - a_2 \bm{n}_f \cdot ( \bm{x} -\bm{c}_f ) - a_2 \bm{n}_f \cdot \bm{c}_f  \\ \eta x_3 + a_1 \bm{n}_f \cdot ( \bm{x} - \bm{c}_f ) + a_1 \bm{n}_f \cdot \bm{c}_f \end{array} \right )
\end{eqnarray*}
where $\eta = \bm{n}_f \cdot (-a_3, a_2, -a_1) = -a_1 n_3 + a_2 n_2 - a_3 n_1$ and $\bm{c}_f$ is the barycenter of the face $f$. 
In case $\eta = 0$, $(\bm{A} \bm{x} \times \bm{n}_f)|_f = \bm{n}_f \cdot \bm{c}_f (a_3, -a_2, a_1)^T.$ Note that $(a_3, -a_2, a_1)^T \in f$ and so, it can be expressed as $\bm{\mu} \times \bm{n}_f$ for some vector $\bm{\mu}$. Therefore, $(\bm{v} \times \bm{n}_f)|_f = \bm{b} \times \bm{n}_f$ with $\bm{b} = \bm{a} + \bm{\mu}$. In case $\eta \neq 0$, we observe that for $\bm{x} \in f$,  
\begin{eqnarray*}
\bm{A} \bm{x} \times \bm{n}_{f} = \eta (\bm{x} - \bm{c}),  
\end{eqnarray*} 
where 
\begin{equation}
\bm{c} =  \frac{1}{\eta} \left ( \begin{array}{c} -a_3 \bm{n}_{f} \cdot \bm{c}_f \\  a_2  \bm{n}_{f} \cdot \bm{c}_f \\ - a_1 {\bm{n}}_{f} \cdot \bm{c}_f \end{array} \right )  \quad \mbox{ and } \quad (\bm{x} - \bm{c})\cdot \bm{n}_f = 0.  
\end{equation}
This means that we can modify $\bm{A}\bm{x} \times \bm{n}_f$ further into the following, with $c = \eta$, 
\begin{equation}
\bm{A} \bm{x} \times \bm{n}_{f} = c (\bm{x} - \bm{c}) = c (\bm{x} - \bm{c}_  f + \bm{c}_f - \bm{c}) = \bm{d} + c (\bm{x} - \bm{c}_f),  
\end{equation}
where $\bm{d}$ is a constant vector such that $\bm{d} \cdot \bm{n}_f = 0$. Therefore, we arrive at the conclusion and this completes the proof. 
\end{proof}
For 3D, we remark that for any given $T \in \mathcal{T}_h$ with $f$ being a face of $T$, the dimension of the space $\{\bm{q}(\bm{x}) \times \bm{n}_f : \bm{q} \in \bm{RM}(T), \,\, \forall \bm{x} \in f\}$ is three and we denote the space $\{ (c_1 + c x_1) \bm{t}_1 + (c_2 + c x_2) \bm{t}_2 : x_1, x_2 \in \Reals{} \}$ by ${\rm RT}_0(f)$ \cite{xie2008uniformly}. Now, we can prove the Korn's inequality for functions in $(H^1(\Omega,\mathcal{T}_h))^d$. We note that the spaces of projection onto the face of each $T$ are different for {\rm 2D} and {\rm{3D}}, the detailed jump terms shall be stated accordingly. We first define $\pi_1$ by the $L^2$ orthogonal projection from $L^2(f)$ onto $P_1(f)$ and $\pi_{RM^\perp(f)}$ is the $L^2$ orthogonal projection form $L^2(f)$ onto $P_0(f)$ for $d = 2$, but it is from $L^2(f)$ to ${{\rm RT}}_0(f)$ for $d = 3$, respectively. 

To state the discrete Korn's inequality, we shall first introduce on each $T \in \mathcal{T}_h$, a projection operator $\Pi_T$ from $(H^1(T))^d$ onto the $\bm{RM}(T)$ by the following conditions:
\begin{equation} 
\left|\int_{T}(\bm v - \Pi_{T}\bm v)dx\right|=0, \quad  \forall \bm v \in (H^1(T))^d,
\end{equation}
\begin{equation} 
\left|\int_{T}\nabla \times(\bm v-\Pi_{T}\bm v) dx\right| = 0, \quad \forall \bm v \in (H^1(T))^d.
\end{equation}
Hence, by the definition of $\Pi_T$, we have (see (3.3) in \cite{brenner2004korn})
\begin{equation} \label{3.3}
| \bm v-\Pi_{T}\bm v |_{H^1(T)} \lesssim  
\| \mathcal{D} (\bm v - \Pi_T \bm v) \|_0 = \|\mathcal{D} (\bm v) \|_0, \quad \forall \bm v \in (H^1(T))^d.
\end{equation}
\begin{equation} \label{3.4}
\|\bm v-\Pi_{T}\bm v\|_0\lesssim  (\hbox{diam} T)|\bm v-\Pi_{T}\bm v|_{H^1(T)}, \quad \forall \bm v\in (H^1(T))^d.
\end{equation}
Using this local projection $\Pi_T$, we define $\Pi: (H^1(\Omega,\mathcal{T}_h))^d\longrightarrow \boldsymbol{V}_h$ by  
\begin{equation}
(\Pi \bm u) = \Pi_T\bm u_T, \quad \forall T \in \mathcal{T}_h.
\end{equation} 
Next, following \cite{brenner2004korn}, we also introduce a seminorm on $(H^1(\Omega;\mathcal{T}_h))^d$, denoted by $\Phi$ satisfying the following Assumptions:
\begin{itemize} 
\item[(C1)] $|\Phi(\bm w)| \lesssim \|\bm w\|_{1}, \quad \forall \bm{w} \in (H^1(\Omega))^d,$ 
\item[(C2)] $\Phi(\bm{m}) = 0$ and $\bm{m} \in {\bf{RM}}(\Omega)$ if and only if $\bm{m}$ is constant. 
\item[(C3)] $(\Phi(\bm v - E \bm v))^2 \lesssim \sum_{f \in \mathcal{E}_h^o}(\hbox{diam}(f))^{d-2}\sum_{p\in  \mathcal{V}(f)}|\jump{\bm v}_f(p)|^2, \, \forall \bm v \in \bm{V}_h$, where $\mathcal{V}(f)$ is the set of the vertices of $f$.
\end{itemize} 
The first estimate is well-known for the piecewise $H^1$ functions, see \cite{brenner2004korn} : 
\begin{Lemma}
Let $\Phi$ be the seminorm on $H^1(\Omega;\mathcal{T}_h)$ satisfying the Assumptios (C1), (C2), and (C3). Then, the following estimate holds:
\begin{equation}\label{2.9}
|\bm v|^2_{H^1(\Omega,\mathcal{T}_h)} \lesssim \|\mathcal{D}_{\mathcal {T}}(\bm v)\|_0^2+
(\Phi(\bm v))^2+\sum_{f \in \mathcal{E}^o_h}(\hbox{diam}~ f)^{d-2}\sum_{p\in  \mathcal{V}(f)}|\jump{\bm v}_f(p)|^2
\end{equation}
for all $\bm v \in \bm{V}_h$, where $\mathcal{D}_{\mathcal {T}}(\bm v)|_T=\mathcal{D}(\bm v|_T)$
for all $T \in \mathcal T_h$. 
\end{Lemma}
Now, in order to refine the discrete Korn's inequality, we further impose an addition Assumption for $\Phi$ as follows: 
\begin{itemize} 
\item[(C4)] $|\Phi(\bm u-\Pi\bm u)| \lesssim \|\mathcal D_\mathcal{T}(\bm u)\|_0, \quad \forall \bm u \in (H^1(\Omega,\mathcal{T}_h))^d$. 
\end{itemize} 
\begin{remark}
For  $u\in (H^1(\Omega))^d$, Assumption (C1) implies Assumption (C4), but for $u \in (H^1(\Omega,\mathcal{T}_h))^d$, Assumption (C4) can not be derived from Assumption (C1). 
\end{remark}
We then obtain the following main result in this paper: 
\begin{Theorem}\label{main}
Let $\Phi:(H^1(\Omega, \mathcal{T}_h))^d\longrightarrow \Reals{}$ be a seminorm satisfying  the Assumptions (C1), (C2),  (C3) and (C4). We have the following results for 2D and 3D, respectively. For $2D$, we have 
\begin{eqnarray*} 
&& |\bm u|^2_{H^1(\Omega,\mathcal{T}_h)} \lesssim \|\mathcal{D}_{\mathcal {T}}(\bm u)\|_0^2 + (\Phi(\bm u))^2 \\ 
&&+ \sum_{f \in \mathcal{E}^o_h}(\hbox{diam}(f))^{-1} \Big(\|\left [\pi_1(\jump{\bm u}_f \cdot \bm n_f) \right ] \bm{n}_f \|_{0,f}^2 + 
\| \left [\pi_{RM^\perp(f)}(\jump{\bm u}_f \cdot \bm t_f) \right ] \bm t_f \|_{0,f}^2\Big). 
\end{eqnarray*}
For $3D$, we have  
\begin{eqnarray*} 
&& |\bm u|^2_{H^1(\Omega,\mathcal{T}_h)} \lesssim \|\mathcal{D}_{\mathcal {T}}(\bm u)\|_0^2 +
(\Phi(\bm u))^2 \\
&&+ \sum_{f \in \mathcal{E}^o_h}(\hbox{diam}(f))^{-1}\Big(\|\left [ \pi_1(\jump{\bm u}_f\cdot \bm n_f) \right ] \bm{n}_f \|_{0,f}^2+\| \left [\pi_{RM^\perp(f)}(\jump{\bm u}_f \times \bm n_f) \right ] \times \bm{n}_f\|_{0,f}^2\Big). 
\end{eqnarray*}
\end{Theorem}
\begin{proof}
Let $\bm u\in (H^1(\Omega,\mathcal{T}_h))^d$, from \eqref{2.9} and \eqref{3.3}, we have 
\begin{equation*}\label{3.8} 
\begin{split}
|\bm u|^2_{H^1(\Omega,\mathcal{T}_h)}&\lesssim |\bm u-\Pi \bm u|^2_{H^1(\Omega,\mathcal{T}_h)}+  |\Pi \bm u|^2_{H^1(\Omega,\mathcal{T}_h)}\\
&\lesssim
\|\mathcal{D}_{\mathcal {T}}(\bm u)\|_0^2+
(\Phi(\Pi\bm u))^2+\sum_{f \in \mathcal{E}^o_h}(\hbox{diam}(f))^{d-2}\sum_{p\in  \mathcal{V}(f)}|\jump{\Pi \bm u}_f(p)|^2.
\end{split}
\end{equation*}
Using condition (C4), we find
\begin{equation}\label{3.8} 
\Phi(\Pi \bm u)\leq \Phi(\bm u-\Pi\bm u)+\Phi(\bm u)\lesssim \|\mathcal{D}_{\mathcal {T}}(\bm u)\|_0+\Phi(\bm u).
\end{equation}
Let $f \in \mathcal{E}^o_h$ be arbitrary and $p \in \mathcal V(f)$, we have, by inverse estimate
\begin{equation}\label{3.9} 
|\jump{\Pi\bm u}_f(p)|^2\lesssim  (\hbox{diam}(f))^{1-d} \|
\jump{\Pi\bm u}_f\|_{0,f}^2.
\end{equation}
We first prove the $d = 2$ case. Let $f = T^{+} \cap T^{-}$ 
and choose $\bm{n}_f = \bm n^{-}$ as the unit normal vector and 
$\bm{t}_f = \bm t^{-}$ as the unit tangential vector of $f$. 
Then we have that 
\begin{equation}
\begin{split}
\|\jump{\Pi\bm u}_f\|_{0,f}^2 &= \int_f |(\jump{\Pi\bm u}_f \cdot \bm{n}_f )\bm{n}_f |^2ds + \int_f|(\jump{\Pi\bm u}_f \cdot \bm{t}_f)\bm{t}_f |^2 ds
 \\
&= \int_f (\jump{\Pi\bm u}_f \cdot \bm{n}_f)^2
+(\jump{\Pi\bm u}_f \cdot \bm{t}_f)^2 ds.
\end{split}
\end{equation}
Using Theorem \ref{kornpre}, we see that 
\begin{equation} 
\jump{\Pi\bm u}_f \cdot \bm{t}_f = c
\end{equation} 
for some constant $c$. Therefore, we have that 
\begin{equation}\label{3.10} 
\begin{split}
\|\jump{\Pi\bm u}_f\|_{0,f}^2&=\int_f \Big(\pi_1(\jump{\Pi\bm u}_f \cdot \bm{n}_f )\Big)^2 \, ds + \int_f \Big(\pi_{RM^\perp(f)}(\jump{\Pi\bm u}_f\cdot \bm{t}_f)\Big)^2\, ds \\
&\leq \int_f \Big(\pi_1(\jump{\Pi \bm u-\bm u}_f\cdot \bm{n}_f )\Big)^2 \, ds +\int_f \Big(\pi_{RM^\perp(f)}(\jump{\Pi \bm u-\bm u}_f \cdot \bm{t}_f )\Big)^2 \, ds\\
&+ \int_f \Big(\pi_1(\jump{\bm u}_f \cdot \bm{n}_f )\Big)^2 \, ds +\int_f \Big(\pi_{RM^\perp(f)}(\jump{\bm u}_f \cdot \bm{t}_f )\Big)^2 \, ds.
\end{split}
\end{equation}
Let $\mathcal T_f = T^{+}\cup T^{-}$, it follows from \eqref{3.3}, \eqref{3.4} and trace theorem that 
\begin{equation}\label{3.11} 
\begin{aligned}
 &\int_f \Big(\pi_1(\jump{\Pi\bm u-\bm u}_f \cdot \bm{n}_f )\Big)^2 \, ds +\int_f \Big(\pi_{RM^\perp(f)}(\jump{\Pi\bm u-\bm u}_f \cdot \bm{t}_f )\Big)^2 ds \\ 
  \leq&\|\jump{\Pi\bm u-\bm u}\|^2_{0,f} \\
   \lesssim & \sum_{T \in \mathcal T_f} 
\Big((\hbox{diam}(T))|\bm u_T - \Pi_T\bm u_T|_{H^1(T)}^2 + (\hbox{diam}(T))^{-1}|\bm u_T-\Pi_T\bm u_T|_{0,T}^2\Big) \\
 \lesssim &\sum_{T \in \mathcal{T}_f} (\hbox{diam}(T))\|\mathcal D(\bm u_T)\|_{0,T}^2.
\end{aligned}
\end{equation}
Combining \eqref{3.9}  \eqref{3.10} and \eqref{3.11}, and noting that diameter of $T$ is equivalrent to diameter of $f$, we find
\begin{equation}\label{3.12} 
\begin{split}
&\sum_{f \in \mathcal{E}^o_h}(\hbox{diam}(f))^{d-2} \sum_{p \in  \mathcal{V}(f)}|\jump{\bm u}_f(p)|^2\\
&\lesssim \|\mathcal D_{\mathcal T}(\bm u)\|_{0}^2+\sum_{f \in \mathcal{E}^o_h} (\hbox{diam}(f))^{-1}\Big(\|\pi_1(\jump{\bm u}_f \cdot \bm{n}_f)\|_{0,f}+\|\pi_{RM^\perp(f)}(\jump{\bm u}_f \cdot \bm{t}_f)\|_{0,f}\Big).
\end{split}
\end{equation}
For $d=3$, we have that 
\begin{equation} 
\jump{\Pi \bm u}_f = -(\jump{\Pi\bm u}_f \times \bm{n}_f ) \times \bm{n}_f + (\jump{\Pi\bm u}_f \cdot \bm{n}_f ) \bm{n}_f. 
\end{equation} 
Furthermore, similar to the proof of Theorem \ref{kornpre}, we can show that 
\begin{equation} 
\jump{\Pi\bm u}_f \times \bm{n}_f = \bm{a} \times \bm{n}_f + c (\bm{x} - \bm{c}_f), \quad \forall \bm{x} \in f, 
\end{equation} 
where $\bm{c}_f$ is a barycenter of $f$.  Using the same argument for $2D$, we can establish the estimate for 3D case in Theorem \ref{main}. This completes the proof. 
\end{proof}
\begin{definition}
A subspace of rigid body mode will be denoted and defined by 
\begin{equation} 
\bm {RM}^{\partial}(\Omega) = \left \{\bm m \in \bm {RM}(\Omega): ~ \|\bm m\|_{L^2(\partial \Omega)}=1,\int_{\partial \Omega}\bm mds=0 \right \},
\end{equation} 
\end{definition} 
We shall set $\Phi(\bm u)$ as follows: 
\begin{equation}
\Phi(\bm u) : = \sup \limits_{\bm m \in \bm {RM}^{\partial} (\Omega)} \int_{\partial \Omega}\bm u \cdot \bm m \, ds. 
\end{equation} 
Under this setting, we can show that $\Phi$ satisfies the Aummptions (C1), (C2), (C3) and (C4). Additionally, we have 
\begin{equation} 
\Phi(\bm u) \leq \sum_{f \subset \partial \Omega}\Big(\|\pi_1(\jump{\bm u}_f \cdot \bm n_f)\|_{0,f}^2+\|\pi_{RM^\perp(f)}(\jump{\bm u}_f \cdot \bm t_f)\|_{0,f}^2\Big)~\hbox{for}~ d=2,
\end{equation} 
and 
\begin{equation} 
\Phi(\bm u)\leq \sum_{f \subset \partial \Omega}\Big(\|\pi_1(\jump{\bm u}_f\cdot \bm n_f)\|_{0,f}^2+\|\pi_{RM^\perp(f)}(f)(\jump{\bm u}_f \times \bm n_f)\|_{0,f}^2\Big)~\hbox{for}~ d=3.
\end{equation} 
Let $P_{1/0}(f) = P_1(f)/P_0(f)$ and $\pi_{1/0}$ be the orthogonal projection from $L^2(f)$ onto $P_{1/0}(f)$.
\begin{Corollary}\label{DisKorn2d}
For any $\bm u\in (H^1(\Omega, \mathcal{T}_h))^2$, and if $\bm u$ satisfies the following continuity conditions across the edges $f\in \mathcal{E}_h^o$
\begin{enumerate}
\item $\int_{f} \jump{\bm u}_f \cdot \bm n_f \, p \, ds = 0, \quad \forall p \in P_1(f)$;
\item $\int_{f} \jump{\bm u}_f \cdot \bm t_f ds = 0$;
\end{enumerate}
then the following estimate holds:
\begin{equation} 
|\bm u|_{H^1(\Omega,\mathcal{T}_h)}\lesssim \|\mathcal{D}_{\mathcal {T}}(\bm u)\|_0+\sup\limits_{\bm m\in \bm {RM}^\partial(\Omega)}\int_{\partial \Omega}\bm u\cdot \bm mds.
\end{equation}
\end{Corollary}
\begin{proof} 
This is immediate from Theorem \ref{main}. This completes the proof. 
\end{proof} 

\begin{Corollary}\label{DisKorn3d}
For any $\bm u \in (H^1(\Omega, \mathcal{T}_h))^3$, and if $\bm u$ satisfies the following continuity conditions across the edges $f \in \mathcal{E}_h^o$
\begin{enumerate}
\item $\int_{f} \jump{\bm u}_f \cdot \bm n_f \, p \, ds = 0, \quad \forall p \in P_1(f)$;
\item $\int_{f} (\jump{\bm u}_f \times \bm n_f ) 
\cdot \bm {p}\, ds = 0, \quad \forall \bm{p} \in 
{\rm RT}_0(f)$;
\end{enumerate}
then the following estimate holds:
\begin{equation} 
|\bm u|_{H^1(\Omega,\mathcal{T}_h)}\lesssim \|\mathcal{D}_{\mathcal {T}}(\bm u)\|_0+\sup\limits_{\bm m\in \bm {RM}^\partial(\Omega)}\int_{\partial \Omega}\bm u\cdot \bm m\, ds.
\end{equation}
\end{Corollary}
\begin{proof} 
This is immediate from Theorem \ref{main}. This completes the proof. 
\end{proof} 

\section{On sharpness of the Korn's inequality}\label{sharp}

In this subsection, we establish that the proposed Korn's inequality is sharp. We first start with 2D case and move on to 3D case. The goal is to establish that the proposed Korn's inequality is sharp, see Corollary \ref{DisKorn2d} and Corollary \ref{DisKorn3d}. More precisely, the above continuity conditions in Corollary \ref{DisKorn2d} and Corollary \ref{DisKorn3d} minimize the conditions to obtain the classical Korn's inequality for piecewise $H^1$ space. To put it in another way, if one of the conditions is violated, the classical Korn's inequality does not hold, but the inequality can be made to hold by adding an appropriate jump terms. 

For a sharpness proof, we shall take examples showing that if any one of the conditions is missing, then the inequality does not hold. These examples are constructed in special 2D and 3D domains as given in Figure \ref{ex}. 
\begin{figure}[h] 
\includegraphics[width=12cm, height=5cm]{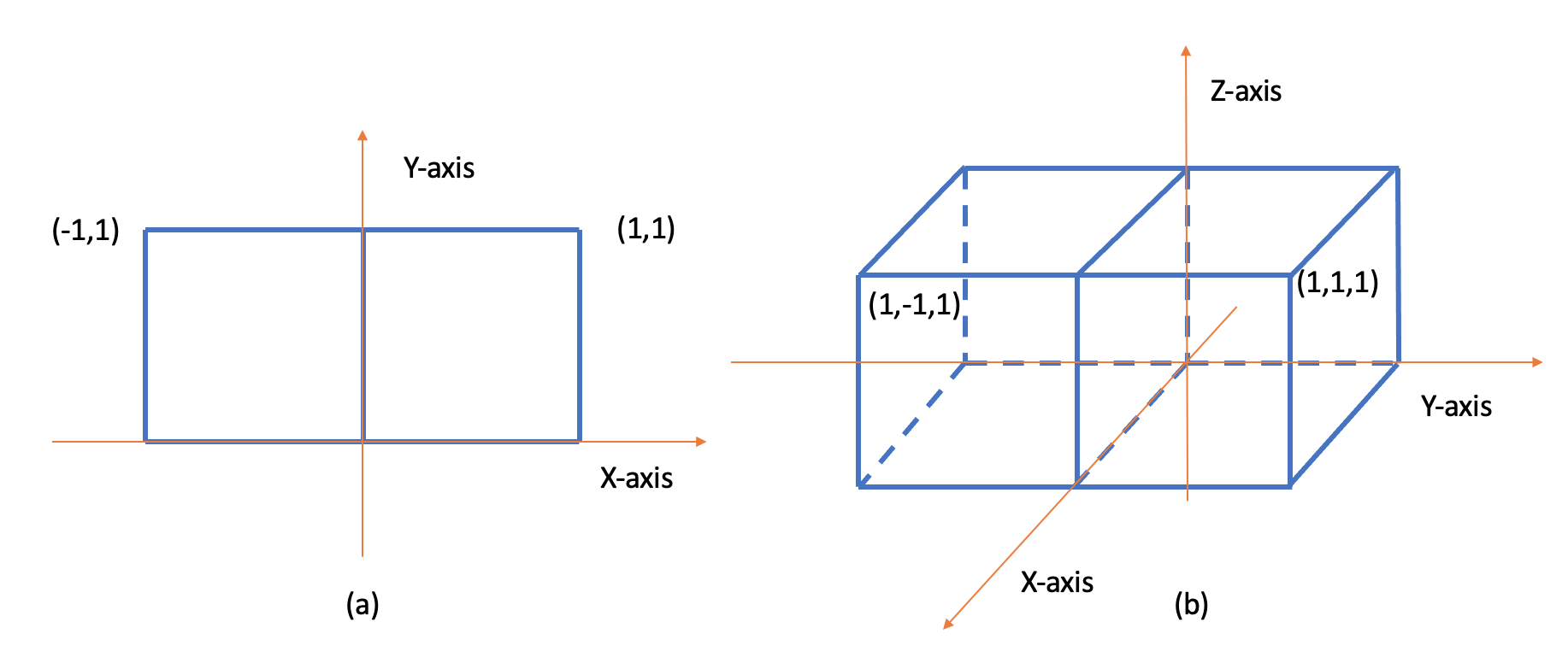}
\caption{Domain 2D and 3D}\label{ex} 
\end{figure} 
For simplicity, we shall divide our discussion for 2D and 3D. For 2D, we consider a special $\Omega=\mathcal T_h = T_1 \cup T_2$, which consists of the vertexes $(-1,0), (0,0), (0,1),(-1,1)$ and $ (0,0),(1,0), (1,1),(0,1)$. We denote $f = T_1 \cap T_2$ and 
\begin{subeqnarray*} 
E_1 &=& \{\bm u\in (H^1(\Omega, \mathcal{T}_h))^2: \pi_{1/0}(\jump{\bm u}_f \cdot \bm n_f)=0~\hbox{and}~\pi_0(\jump{\bm u}_f \cdot \bm t_f)=0\}, \\
E_2 &=& \{\bm u\in (H^1(\Omega, \mathcal{T}_h))^2: \pi_0(\jump{\bm u}_f \cdot \bm n_f)=0~\hbox{and}~\pi_0(\jump{\bm u}_f \cdot \bm t_f)=0\}, \\ 
E_3 &=& \{\bm u \in (H^1(\Omega, \mathcal{T}_h))^2: \pi_0(\jump{\bm u}_f \cdot \bm n_f)=0~\hbox{and}~\pi_{1/0}(\jump{\bm u}_f \cdot \bm n_f)=0\}.
\end{subeqnarray*} 
We first note that a simple calculation leads that 
\begin{equation}\label{RM0} 
\bm {RM}^\partial(\Omega) \subset \hbox{span}\{(1-2y, 2x)^t\}. 
\end{equation} 
Furthermore, we shall note that the definitions of $\pi_0$ and $\pi_{1/0}$ mean 
\begin{eqnarray}
\pi_0 (u) &=& \frac{1}{|f|} \int_f u ds \quad \mbox{ and } \quad 
\pi_{1/0}(u) = \frac{1}{|f|} \int_f u s ds. 
\end{eqnarray} 
\begin{Theorem}
There exists $\bm u\in E_k,$ for $k=1,2,$ or $3$, such that 
 \begin{equation}\label{notkorn1} 
 |\bm u|^2_{H^1(\Omega,\mathcal{T}_h)} \neq 0, \quad \mbox{ but }  \quad 
 \|\mathcal{D}_{\mathcal {T}}(\bm u)\|_0^2 + \left ( \sup\limits_{\bm m\in \bm {RM}^\partial (\Omega)}\int_{\partial \Omega}\bm u\cdot \bm mds \right )^2=0.
 \end{equation} 
\end{Theorem}
\begin{proof}
Let us consider $\bm u$ such that $\bm u_i = \bm u|_{T_i}$ for $i = 1,2$ given as follows:  
\begin{equation}\label{form} 
\bm u_i = \left ( \begin{array}{c} a_i \\ b_i \end{array} \right ) + c_i\left ( \begin{array}{c} -y \\ x \end{array} \right ) \in \bm {RM}(T_i), \quad\forall i=1,2. 
\end{equation}
Here six coefficients $(a_i,b_i,c_i)_{i=1,2}$ shall be appropriately determined so that the Korn's inequality does not hold. We observe that 
$\|\mathcal{D}_{\mathcal {T}}(\bm u)\|_0 = 0$ by construction for any choice of coefficients. We first note that from \eqref{RM0}, i.e., $\bm {RM}^\partial(\Omega) \subset\hbox{span}\{(1-2y, 2x)^t\}$, we can calculate that with $\bm m = c(1 - 2y, 2x)^t$ and $c \neq 0$, 
\begin{equation}\label{cal} 
\int_{\partial \Omega}\bm u\cdot \bm m \, ds = c\left ( 4 (b_2 - b_1) + \frac{9}{2}(c_1 + c_2) \right ).
\end{equation}
We also observe that with $\overline{n} = n_2 - n_1$, 
\begin{eqnarray*}
\jump{\bm{u}}_f \cdot \bm{n}_f &=&  \overline{a} - \overline{c}y \\ 
\jump{\bm{u}}_f \cdot \bm{t}_f &=& \overline{b} + \overline{c} x. 
\end{eqnarray*} 
A simple computation leads that
\begin{eqnarray*}
\int_f \jump{\bm{u}}_f \cdot \bm{n}_f ds &=& \int_0^1 \overline{a} - \overline{c} y \, dy = \overline{a} - \frac{1}{2} \overline{c} \\
\int_f (\jump{\bm{u}}_f \cdot \bm{n}_f) y ds &=& \int_0^1 \overline{a}y - \overline{c} y^2 \, dy =  \frac{1}{2} \overline{a} - \frac{1}{3} \overline{c} \\ 
\int_f \jump{\bm{u}}_f \cdot \bm{t}_f \, ds &=& \int_0^1 \overline{b} \, dy = \overline{b}. 
\end{eqnarray*} 
Now, we begin our search of $\bm{u}$ of the aforementioned form \eqref{form} that belongs to $E_k$, which satisfies \eqref{notkorn1} for all $k=1,2,$ or $3$. This shall be discussed case by case as follows. 
\begin{itemize}[leftmargin=*]
\item For $k=1$, we shall choose $\overline{a} = \frac{2}{3} \overline{c} \neq 0$ and $\overline{b} = 0$. This way, we can make $\bm u \in E_1$. On the other hand, for \eqref{cal}, we must choose $c_1 + c_2 = 0$. Then, $\bm u$ satisfies \eqref{notkorn1}. 
\item For $k=2$, we shall choose $\overline{a} = \frac{1}{2} \overline{c} \neq 0$. Then by choosing $\overline{b} = 0$, we can make $\bm u \in E_2$. Again, we choose $c_1 + c_2 = 0$ for \eqref{cal}. Namely, $c_1 = c/2$ and $c_2 = -c/2$ for arbitrary $c \neq 0$ to guarantee the above conditions. We note that for $c \neq 0$, $\bm u$ satisfies \eqref{notkorn1}. 
\item For $k=3$, we shall choose $\overline{a} = \overline{c} = 0$. Then by choosing $\overline{b} \neq 0$, we can make $\bm u \in E_3$. On the other hand, for \eqref{cal}, we choose $c_1 =- c_2 = c/2\neq 0$ so that it holds, we see that since $c \neq 0$, $\bm u$ satisfies \eqref{notkorn1}. 
\end{itemize} 
This completes the proof.
\end{proof}
We shall now turn our attention to three dimensional case. For 3D case, we consider two cubes as given in Figure \ref{ex} (b). Namely, $\Omega=\mathcal T_h = T_1 \cup T_2$, whose coordinates are $(1, -1, 0), (1, 0, 0)$,  $(0, 0, 0),(0, -1, 0),(1, -1, 1), (1, 0, 1),(0, 0, 1),(0, -1, 1)$ 
and $(1, 0, 0), (1, 1, 0),(0, 1, 0)$, $(0, 0, 0), (1, 0, 1), (1, 1, 1),(0, 1, 1),(0, 0, 1)$.  We denote $f = T_1\cap T_2$ by the face and expand the two conditions in Corollary \ref{DisKorn3d} as following six conditions 
\begin{enumerate}
\item [(A1)]   $\int_{f} \jump{\bm u}_f \cdot \bm n_f ds = 0$;
\item [(A2)]   $\int_{f} \jump{\bm u}_f \cdot \bm n_f x ds = 0$;
\item [(A3)]   $\int_{f} \jump{\bm u}_f \cdot \bm n_f z ds = 0$;
\item  [(A4)]   $\int_{f} (\jump{\bm u}_f \times \bm n_f) \cdot (0,0,1)^T ds = 0$;
\item  [(A5)]   $\int_{f} (\jump{\bm u}_f \times \bm n_f) \cdot (1,0,0)^T ds = 0$;
\item  [(A6)]   $\int_{f} (\jump{\bm u}_f \times \bm n_f) \cdot (x,0,z)^T ds = 0$.
\end{enumerate}
We now list a total of six subsets of $(H^1(\Omega;\mathcal{T}_h))^3$ as follows: 
\begin{subeqnarray*} 
F_1 &=& \{\bm u\in (H^1(\Omega, \mathcal{T}_h))^3: \bm u ~~ \hbox{satisfies (A2), (A3), (A4), (A5), and (A6)} \} \\
F_2 &=& \{\bm u\in (H^1(\Omega, \mathcal{T}_h))^3: \bm u ~~ \hbox{satisfies (A1), (A3), (A4), (A5), and (A6)} \} \\
F_3 &=& \{\bm u\in (H^1(\Omega, \mathcal{T}_h))^3: \bm u ~~ \hbox{satisfies (A1), (A2), (A4), (A5), and (A6)} \} \\
F_4 &=& \{\bm u\in (H^1(\Omega, \mathcal{T}_h))^3: \bm u ~~ \hbox{satisfies (A1), (A2), (A3), (A5), and (A6)} \} \\
F_5 &=& \{\bm u\in (H^1(\Omega, \mathcal{T}_h))^3: \bm u ~~ \hbox{satisfies (A1), (A2), (A3), (A4), and (A6)} \} \\
F_6 &=& \{\bm u\in (H^1(\Omega, \mathcal{T}_h))^3: \bm u ~~ \hbox{satisfies (A1), (A2), (A3), (A4), and (A5)} \}. 
\end{subeqnarray*} 
We first establish that with $\Omega = T_1\cup T_2$, the space of $\bm {RM}^\partial(\Omega)$. A tedious but simple calculation shows that 
\begin{equation}\label{RM1} 
\bm{RM}^\partial(\Omega) \subset \hbox{span}\{\bm m_1, \bm m_2, \bm m_3\}, 
\end{equation} 
where $\bm m_i$ with $i=1,2,3$ are given as follows: 
\begin{eqnarray*}
\bm m_1 = \left ( \begin{array}{c} -1 + 2z \\ 0 \\ 1 - 2x \end{array} \right ), \,\, \bm m_2 = \left ( \begin{array}{c} 2y \\ 1 - 2x \\ 0 \end{array} \right ), \,\, \mbox{ and } \,\, \bm m_3 = \left (\begin{array}{c} y \\ -x + z \\ -y \end{array} \right ).  
\end{eqnarray*}
%
We shall now state and prove the main result in this section. 
\begin{Theorem}
There exists $\bm u \in F_k,$ for $k=1,2,3,4,5$ or $6$, such that 
\begin{equation}\label{notkorn2} 
|\bm u|^2_{H^1(\Omega,\mathcal{T}_h)} \neq 0, \quad \mbox{ but }  \quad 
 \|\mathcal{D}_{\mathcal {T}}(\bm u)\|_0^2 + \left ( \sup\limits_{\bm m\in \bm {RM}^\partial (\Omega)}\int_{\partial \Omega}\bm u\cdot \bm mds \right )^2=0.
\end{equation} 
\end{Theorem}
\begin{proof} 
Let us consider $\bm u$ such that $\bm u_i = \bm u|_{T_i}$ for $i=1,2$ given as follows: 
\begin{equation}\label{form} 
\bm u_i = \left ( \begin{array}{c} a_i \\ b_i \\ c_i \end{array} \right ) + \left ( \begin{array}{ccc} 0 & d_i & e_i \\ -d_i & 0 & f_i \\ -e_i & -f_i & 0 \end{array} \right ) \left ( \begin{array}{c} x \\ y \\ z \end{array} \right ) \in \bm {RM}(T_i), \quad\forall i=1,2. 
\end{equation}
Here twelve coefficients $(a_i,b_i,c_i,d_i,e_i,f_i)_{i=1,2}$ shall be appropriately determined so that the Korn's inequality does not hold under the condition that $\bm u \in F_k$ for any $k$. We observe that 
$\|\mathcal{D}_{\mathcal {T}}(\bm u)\|_0 = 0$ by construction for any choice of coefficients. We first investigate what constraints are given from the conditions:  
\begin{equation}
\int_{\partial \Omega} \bm u \cdot \bm m_i \, ds = 0, \quad \forall i=1,2,3,  
\end{equation} 
where $\bm {RM}^\partial(\Omega) \subset\hbox{span}\{ \bm m_1, \bm m_2, \bm m_3\}$. We observe that 
\begin{eqnarray*}
\bm u \cdot \bm m_1|_{T_i} &=& (a_i + d_i y + e_i z) ( -1 + 2z ) + (c_i - e_i x - f_i y) (1 - 2x); \\
\bm u \cdot \bm m_2|_{T_i} &=& (a_i + d_i y + e_i z) ( 2y ) + (b_i - d_i x + f_i z) ( 1 - 2x ); \\ 
\bm u \cdot \bm m_3|_{T_i} &=& (a_i + d_i y + e_i z) ( y ) + 
(b_i - d_i x + f_i z) ( -x + z ) + (c_i - e_i x - f_i y) (- y).
\end{eqnarray*} 
A simple but tedious computation leads that 
\begin{subeqnarray}\label{rot} 
\int_{\partial \Omega} \bm u \cdot \bm m_1 \, ds &=& 
\left ( \frac{9}{3} (e_1 + e_2)  \right ) ;\\ 
\int_{\partial \Omega} \bm u \cdot \bm m_2 \, ds &=& 6(a_2 - a_1) + 3(e_2  - e_1) + \frac{37}{6}(d_1 + d_2) \slabel{eq2};\\ 
\int_{\partial \Omega} \bm u \cdot \bm m_3 \, ds &=& 3(a_2 - a_1)  - 3(c_2 - c_1) + 3 (e_2 - e_1) \slabel{eq3} \\
&& + \frac{37}{12} \left [ (d_2 + d_1)  + (f_2 + f_1) \right ]. \nonumber    
\end{subeqnarray} 
We also observe that with $\overline{n} = n_2 - n_1$, 
\begin{eqnarray*}
\jump{\bm{u}}_f \cdot \bm{n}_f &=&  \overline{b} - \overline{d} x + \overline{f}z; \\ 
\jump{\bm{u}}_f \times \bm{n}_f &=& (- (\overline{c} - \overline{e} x - \overline{f} y), 0, \overline{a} + \overline{d} y + \overline{e} z)^t|_f = (- (\overline{c} - \overline{e} x), 0, \overline{a} + \overline{e} z)^t.
\end{eqnarray*} 
A simple computation leads that
\begin{eqnarray*}
\int_f \jump{\bm{u}}_f \cdot \bm{n}_f ds &=& \int_0^1 \int_0^1 \overline{b} - \overline{d} x + \overline{f} z \, dxdz = \overline{b} - \frac{1}{2} \overline{d} + \frac{1}{2} \overline{f}; \\ 
\int_f \jump{\bm{u}}_f \cdot \bm{n}_f x ds &=& \int_0^1 \int_0^1 \overline{b}x  - \overline{d} x^2 + \overline{f} xz \, dxdz =  \frac{1}{2} \overline{b} - \frac{1}{3} \overline{d} + \frac{1}{4} \overline{f}; \\
\int_f \jump{\bm{u}}_f \cdot \bm{n}_f z ds &=& \int_0^1 \int_0^1 \overline{b} z - \overline{d} xz + \overline{f} z^2 \, dxdz = \frac{1}{2} \overline{b} - \frac{1}{4} \overline{d} + \frac{1}{3} \overline{f}.
\end{eqnarray*} 
Furthermore, we have that
\begin{eqnarray*}
\int_f [\jump{\bm{u}}_f \times \bm{n}_f ] \cdot (0,0,1)^T ds &=& \int_0^1 \int_0^1 \overline{a} + \overline{e} z \, dxdz = \overline{a} + \frac{1}{2} \overline{e};\\ 
\int_f [\jump{\bm{u}}_f \times \bm{n}_f ] \cdot (1,0,0)^T ds &=& \int_0^1 \int_0^1 -\overline{c} + \overline{e} x \, dxdz = -\overline{c} + \frac{1}{2} \overline{e};\\
\int_f [\jump{\bm{u}}_f \times \bm{n}_f ] \cdot (x,0,z)^T ds &=& \int_0^1 \int_0^1  -\overline{c} x + \overline{e} x^2 + \overline{a} z + \overline{e} z^2 \, dxdz \\
&=& -\frac{1}{2} \overline{c} + \frac{2}{3} \overline{e} + \frac{1}{2} \overline{a}.  
\end{eqnarray*} 
Now, we begin our search of $\bm{u}$ of the aforementioned form \eqref{form} that belongs to $F_k$, which satisfies \eqref{notkorn2} for all $k=1,2,3,4,5,$ or $6$. This shall be discussed case by case as follows. 
\begin{itemize}[leftmargin=*]
\item For $k=1$, we shall choose $\overline{d} = - \overline{f} \neq 0$. Then by choosing $\overline{b} = -\frac{7}{6} \overline{f}$, we can make $(A2)$ and $(A3)$ hold, but $(A1)$ does not. For $(A4), (A5)$ and $(A6)$, we can simply choose $\overline{a} = \overline{c} = \overline{e} = 0$. On the other hand, for \eqref{rot}, we must choose $d_i, e_i$ and $f_i$ so that $d_1 + d_2 = f_1 + f_2 = e_1 + e_2 = 0$. This means, $e_1 = e_2 = 0$. On the other hand, we can choose $d_2 = c/2, d_1 = -c/2$ and $f_2 = -c/2$ and $f_1 = c/2$ for arbitrary $c \neq 0$ to guarantee the above conditions. We note that for $c \neq 0$, $\bm u$ satisfies \eqref{notkorn2}. 
\item For $k=2$, we shall choose $\overline{b} = \frac{1}{2} \overline{d} \neq 0$. Then by choosing $\overline{f} = 0$, we can make $(A1)$ and $(A3)$ hold, but $(A2)$ does not. For $(A4), (A5)$ and $(A6)$, we can simply choose $\overline{a} = \overline{c} = \overline{e} = 0$. On the other hand, for \eqref{rot}, we must choose $d_i, e_i$ and $f_i$ so that $d_1 + d_2 = f_1 + f_2 = e_1 + e_2 = 0$. This means, $e_1 = e_2 = 0$. On the other hand, we can choose $d_2 = c/2, d_1 = -c/2$ for arbitrary $c \neq 0$ to guarantee the above conditions. We note that for $c \neq 0$, $\bm u$ satisfies \eqref{notkorn2}. 
\item For $k=3$, we shall choose $\overline{b} = \frac{1}{2} \overline{f} \neq 0$. Then by choosing $\overline{d} = 0$, we can make $(A1)$ and $(A2)$ hold, but $(A3)$ does not. For $(A4), (A5)$ and $(A6)$, we can simply choose $\overline{a} = \overline{c} = \overline{e} = 0$. On the other hand, for \eqref{rot}, we must choose $d_i, e_i$ and $f_i$ so that $d_1 + d_2 = f_1 + f_2 = e_1 + e_2 = 0$. This means, $e_1 = e_2 = 0$. On the other hand, we can choose $f_2 = c/2, f_1 = -c/2$ for arbitrary $c \neq 0$ to guarantee the above conditions. We note that for $c \neq 0$, $\bm u$ satisfies \eqref{notkorn2}. 
\item For $k=4$, we shall choose $\overline{c} = \frac{1}{2} \overline{e} \neq 0$ 
and $\overline{a} = -\frac{1}{6} \overline{e} \neq 0$. Then by choosing $\overline{d} = 0$, we can make $(A5)$ and $(A6)$ hold, but $(A4)$ does not. For $(A1), (A2)$ and $(A3)$, we can simply choose $\overline{b} = \overline{d} = \overline{f} = 0$. On the other hand, for \eqref{rot}, we must choose $d_i, e_i$ and $f_i$ so that $e_1 + e_2 = 0$ and both $d_1 + d_2$ and $f_1 + f_2$ appropriately to guarantee \eqref{eq2} and \eqref{eq3}. With such a choice, $\bm u$ satisfies \eqref{notkorn2}.
\item For $k=5$, we shall choose $\overline{c} = \frac{1}{6} \overline{e} \neq 0$ 
and $\overline{a} = -\frac{1}{2} \overline{e} \neq 0$. Then by choosing $\overline{d} = 0$, we can make $(A4)$ and $(A6)$ hold, but $(A5)$ does not. For $(A1), (A2)$ and $(A3)$, we can simply choose $\overline{b} = \overline{d} = \overline{f} = 0$. On the other hand, for \eqref{rot}, we must choose $d_i, e_i$ and $f_i$ so that $e_1 + e_2 = 0$ and both $d_1 + d_2$ and $f_1 + f_2$ appropriately to guarantee \eqref{eq2} and \eqref{eq3}. With such a choice, $\bm u$ satisfies \eqref{notkorn2}.
\item For $k=6$, we shall choose $\overline{c} = \frac{1}{2} \overline{e} = -\overline{a} \neq 0$. Then by choosing $\overline{d} = 0$, we can make $(A4)$ and $(A5)$ hold, but $(A6)$ does not. For $(A1), (A2)$ and $(A3)$, we can simply choose $\overline{b} = \overline{d} = \overline{f} = 0$. On the other hand, for \eqref{rot}, we must choose $d_i, e_i$ and $f_i$ so that $e_1 + e_2 = 0$ and both $d_1 + d_2$ and $f_1 + f_2$ appropriately to guarantee \eqref{eq2} and \eqref{eq3}. With such a choice, $\bm u$ satisfies \eqref{notkorn2}.
\end{itemize} 
This completes the proof. 
\end{proof} 
\section{Some Applications}\label{appl} 

We begin this section with a simple example of the nonconforming finite element space on triangular partition 
introduced in Mardal et al., \cite{mardal2002robust} for the 2D case. Note that the space 
is composed of functions which are cubic vector fields with constant divergence on each element and with linear normal component on each edge. Namely, the local space is given as follows: 
\begin{equation} 
\bm V(T) = \{\bm v \in (P_3(T))^2 : \, {\rm div} \bm v \in P_0(T),\,\, \bm v\cdot\bm n_f |_f \in P_1(f), \,\, \forall f \in \partial T \}. 
\end{equation}  
Note that the degrees of freedom consist of 
\begin{equation}
\int_f (\bm v \cdot \bm n_f) q \, ds, \quad \forall q \in P_1(f) \quad \mbox{ and } \quad \int_f \bm v \cdot \bm t_f \, ds.
\end{equation} 
This is exactly what is required minimally for the Korn's inequality. 

In the next subsection, we shall consider enriched $H({\rm div})$ conforming finite element spaces due to Xue et al. \cite{xie2008uniformly}. We then slightly modify these spaces so that we obtain enriched Crouzeix-Raviart element spaces that satisfy the Korn's inequality. Finally, we shall discuss the lowest enrichment for the $H({\rm div})$ conforming finite element space \cite{johnny2012family} and its remedy by modifying the enrichment. Throughout the section, we consider triangular (for 2D) or tetrahedral (for 3D) partitions, for $T \in \mathcal{T}_h$, ${\rm{BDM}}_\ell(T)$ denotes the local Brezzi-Douglas-Marini (BDM) space of order $\ell \geq 1$, namely
\begin{equation}
{\rm{BDM}}_\ell(T) = (P_{\ell}(T))^d, \quad \ell \geq 1,
\end{equation}
and ${\rm{RT}}_\ell(T)$ denotes the Raviart-Thomas space of order $\ell \geq 0$, 
\begin{equation}
{\rm{RT}}_\ell(T) = (P_{\ell}(T))^d + \widetilde{P}_{\ell}(T) \bm x, \quad \ell \geq 0,  
\end{equation}
where $\widetilde{P}_\ell(T)$ denotes the homogeneous polynomial space of degree $\ell$. Also, for any given $T \in \mathcal{T}_h$, we denote $\lambda_i$ with $i = 1,\cdots,d+1$ the barycentric coordinates of $T$.  

We shall frequently use the following standard bubble functions. Namely, for $T \in \mathcal{T}_h$, we denote $b_T$ the bubble function defined by 
\begin{equation}\label{tbubble} 
b_T = \Pi_{i=1}^{d+1} \lambda_i.
\end{equation} 
Similarly, we can define edge/face bubble functions denoted by $b_f$ for any $f \in \partial T$.   

\subsection{Enriched $H({\rm div})$ conforming finite elements that satisfy the Korn's inequality}
In this section, we shall recall some of finite elements, which can be shown to satisfy the Korn's inequality within our framework. The finite element spaces listed in this subsection are constructed based on ${H}({\rm div};\Omega)$ finite element spaces plus divergence free functions \cite{xie2008uniformly}. This subsection is an example indicating the powerfulness of our framework to show the Korn's inequality. 

The local space of the finite elements introduced in \cite{xie2008uniformly} would be of the following form: 
for $T \in \mathcal{T}_h$, 
\begin{equation}
\bm{V}_1(T) = \bm{V}_{1,T}^0 + \bm{curl} (b_T \bm{Y}),
\end{equation} 
where $\bm{V}^0_{1,T}$ is the following well-known $H({\rm div})$-conforming finite element spaces, either ${\rm{BDM}}_1(T)$ or ${\rm{RT}}_1(T)$. For the space $\bm{Y}$, we choose the following polynomial spaces: 
\begin{equation}
\bm{Y} = \left \{ \begin{array}{ll} 
\bm{Y}_1 = P_1(T) & \mbox{ for } d = 2; \\
\bm{Y}_2 = (P_1(T))^3 & \mbox{ for } d = 3;\\
\bm{Y}_3 = (P_1(T))^3/ {\rm span} \left \{ \left( \lambda_i - \frac{1}{3} \right )\nabla \lambda_i \right \}_{i=1}^4 & \mbox{ for } d = 3. 
\end{array} \right.  
\end{equation}
Total of six elements can be listed in the following Table \ref{xutable} with degrees of freedom. 
\begin{table} 
\begin{tabular}{ |p{2.17cm}||p{1.5cm}|p{0.5cm}|p{5.cm}|p{1cm}| }
\hline
\multicolumn{5}{|c|}{The six modified $H({\rm div})$ elements} \\
\hline\hline
{Elements}  & $\bm{V}_{1,T}^0$ & \bm{Y} & {\rm DOF} & Korn \\
\hline
$1^{\rm st}$ FEM (2D)  & ${\rm RT}_1(T)$     & $\bm{Y}_1$    & $ {\small{\begin{array}{l} 
\int_f \bm{v} \cdot \bm{n} q\, ds, \quad \forall q \in P_1(f), \\ \int_T \bm{v} \cdot \bm{q} \, dx, \quad \bm{q} \in (P_0(T))^2 \\ 
\int_f \bm{v} \cdot \bm{t} \, ds \end{array}}} $ & Yes \\
\hline 
$2^{\rm nd}$ FEM (2D) & ${\rm BDM}_1(T)$  & $\bm{Y}_1$   & $ {\small{\begin{array}{l} 
\int_f \bm{v} \cdot \bm{n} q\, ds, \quad \forall q \in P_1(f), \\  
\int_f \bm{v} \cdot \bm{t} \, ds \end{array}}} $  & Yes \\
\hline
$1^{\rm st}$ FEM (3D)  & ${\rm RT}_1(T)$      & $\bm{Y}_2$      & $ {\small{\begin{array}{l} 
\int_f \bm{v} \cdot \bm{n} q\, ds, \quad \forall q \in P_1(f), \\ \int_T \bm{v} \cdot \bm{q} \, dx, \quad \bm{q} \in (P_0(T))^3 \\ 
\int_f (\bm{v} \times \bm{n})\cdot \bm{r} \, ds, \quad \forall \bm{r} \in {\rm RT}_0(f) \end{array}}} $ & Yes  \\
\hline 
$2^{\rm nd}$ FEM (3D) & ${\rm BDM}_1(T)$  & $\bm{Y}_2$   &  $ {\small{\begin{array}{l} 
\int_f \bm{v} \cdot \bm{n} q\, ds, \quad \forall q \in P_1(f), \\
\int_f (\bm{v} \times \bm{n}) \cdot \bm{r} \, ds, \quad \forall \bm{r} \in {\rm RT}_0(f) \end{array}}} $ & Yes  \\
\hline 
$3^{\rm rd}$ FEM (3D)  & ${\rm RT}_1(T)$      & $\bm{Y}_3$  &$ {\small{\begin{array}{l} 
\int_f \bm{v} \cdot \bm{n} q\, ds, \quad \forall q \in P_1(f), \\ \int_T \bm{v} \cdot \bm{q} \, dx, \quad \bm{q} \in (P_0(T))^3 \\ 
\int_f (\bm{v} \times \bm{n})\cdot\bm{r} \, ds, \quad \forall \bm{r} \in (P_0(f))^2 \end{array}}} $ & No \\
\hline 
$4^{\rm th}$ FEM (3D)  & ${\rm BDM}_1(T)$   & $\bm{Y}_3$   &$ {\small{\begin{array}{l} 
\int_f \bm{v} \cdot \bm{n} q\, ds, \quad \forall q \in P_1(f), \\ 
\int_f (\bm{v} \times \bm{n})\cdot\bm{r} \, ds, \quad \bm{r} \in (P_0(f))^2 \end{array}}} $  & No  \\
 \hline
\end{tabular}\caption{FEMs introduced in \cite{xie2008uniformly}}\label{xutable} 
\end{table} 
Due to our framework, we easily notice that the first four elements are the ones that satisfy the Korn's inequality while the last two do not. 


\subsection{Enriched ${\rm CR}$ finite elements that satisfy the Korn's inequality} 
In this subsection, we shall consider enriched {\rm CR} elements that satisfy the Korn's inequality. Our discussion will be done for 2D and 3D separately. We note that in \cite{falk1991nonconforming}, Falk analyzed Korn's inequality for some nonconforming two dimensional finite element spaces. In particular, it was shown that the Crouzeix-Raviart element \cite{crouzeix1973conforming} does not satisfy such a Korn's inequality. A simple remedy can be accomplished. Namely, we add the same enrichment for $\bm{V}_1(T) = {\rm BDM}_1(T)$ replaced with $({\rm CR}(T))^d$ for both $d = 2$ and $d = 3$ and therefore, we propose the enriched $({\rm{CR}}(T))^d$ finite element spaces so that for $T \in \mathcal{T}_h$, we define the local space and degrees of freedom by 
\begin{equation}
\bm{V}(T) := ({\rm CR}(T))^d + {\bm{curl}}(b_T \bm{Y}), 
\end{equation} 
where $\bm{Y}$ is either $\bm{Y}_1$ or $\bm{Y}_2$, depending on $d = 2$ or $d= 3$, respectively.  
\begin{Lemma}
The space $\bm{V}(T)$ with the degrees of freedom given as in Table \ref{xutable}  is unisolvent. 
\end{Lemma}
\begin{proof} 
We let $\bm{v} \in \bm{V}(T)$ be decomposed into two parts: 
\begin{equation}
\bm{v} = \bm{v}_0 + {\bm{curl}}(b_T \bm{q}), \quad \bm{v}_0 \in ({\rm{CR}}(T))^d, \quad \bm{q} \in \bm{Y}. 
\end{equation} 
In $3D$, we have that
\begin{equation}
{\bm{curl}}(b_T \bm{q})\cdot\bm{n} = {\bm{curl}}_f (b_T \bm{q})_f = 0,  
\end{equation} 
where $(b_T \bm{q})_f$ is the tangential component of $b_T \bm{q}$ on $f$. In 2D, we have that ${\bm{curl}}(b_T q)\cdot\bm{n} = \partial_t (b_T q) = 0$. Thus $\bm{v} \cdot \bm{n} = \bm{v}^0 \cdot \bm{n} \in P_1(f), \quad \forall f \in \partial T$. Since for $({\rm CR}(T))^d$ is unisolvent with the degrees of freedom, $\int_f \bm{v}^0 \cdot \bm{n} q \, ds$ for $q \in P_1(f)$, we arrive at $\bm{v}^0 = \bm{0}$. Therefore, it is sufficient to show that $\bm{curl}(b_T \bm{q}) = \bm{0}$ using the tangential component of the degrees of freedom. However, this part of the proof is essentially presented in \cite{xie2008uniformly} and also in \cite{tai2006discrete}. Therefore, we shall skip the proof and completes the proof. 
\end{proof}  
In passing to next subsection, we propose another enrichment for ${\rm CR}^2$ finite element space that satisfies the Korn's inequality. 
\subsubsection{Enriched {\rm CR} for 2D} 
Let $T \in \mathcal{T}_h$ be a triangle in 2D with $\{a_i, a_j, a_k\}$ vertexes of $T$. We denote the barycentric coordinates by $\{\lambda_i, \lambda_j, \lambda_k\}$. Further, $\bm n_{ij},\bm n_{jk},\bm n_{ki} $ are the normals to the edges $\{e_{ij} = [a_i,a_j], e_{jk} = [a_j,a_k], e_{ki} = [a_k,a_i]\}$. We let ${\rm CR}(T)$ be the standard {\rm CR} element on $T$. We then define an edge bubble function $b$ by 
\begin{equation}
b = \lambda_i \lambda_j + \lambda_j \lambda_k + \lambda_k \lambda_i - \frac{1}{6}, \end{equation}  
where the constant $1/6$ comes from the fact that 
\begin{equation}
\frac{1}{6} = \frac{1}{|f_{\nu\mu}|} \int_{f_{\nu\mu}} \lambda_\nu \lambda_\mu \, ds, \quad \nu\mu = ij, jk, ki. 
\end{equation} 
We are now in a position to define an enriched {\rm CR} element on $T$ by
\begin{equation} 
\bm{E_{CR}}(T) = ({\rm{CR}}(T))^2 + \bm{V_{EC}}(T), 
\end{equation} 
where ${\bm{V_{EC}}}(T)$ is the space with the following functions as bases: 
\begin{equation}
\bm \psi_{ij} = b (\lambda_i-\lambda_j)\bm n_{ij},~~~\bm \psi_{jk}=b (\lambda_j-\lambda_k)\bm n_{jk},~~~\bm \psi_{ki} = b(\lambda_k-\lambda_i)\bm n_{ki}.
\end{equation} 
The degrees of freedom are given as follows: 
\begin{equation}\label{dof} 
\int_{f} \bm v \cdot \bm{t}_f ds \quad \hbox{and} \quad \int_{f} \bm v\cdot \bm{n}_f q  ds, \quad \forall q \in P_1(f).
\end{equation} 
It is simple to prove that the aforementioned degrees of freedom can be equivalently formulated as 
\begin{equation}\label{dof} 
\int_{f} \bm v ds \quad \hbox{and} \quad \int_{f} \bm v \cdot \bm{n}_f q  ds, \quad \forall q \in P_1(f).
\end{equation} 
We shall now prove that the space $\bm{E_{CR}}(T)$ with the degrees of freedom \eqref{dof} is unisolvent. 
\begin{Lemma}
The space $\bm {E_{CR}}(T)$ with the degrees of freedom \eqref{dof} is unisolvent. 
\end{Lemma}
\begin{proof}
We let $\bm{v} \in \bm{E_{CR}}(T)$ be decomposed into two parts: 
\begin{equation}
\bm{v} = \bm{v}_0 + \bm{v}_1, \quad \bm{v}_0 \in ({\rm{CR}}(T))^2, \quad \bm{v}_1 \in \bm{V_{EC}}(T).
\end{equation} 
We shall assume that $\int_f \bm{v} \, ds = 0$ and $\int_{f} \bm v \cdot \bm{n}_f q \, ds=0, \,\, \forall q \in P_1(f)$, for all $f \in \partial T$ and shall show that $\bm v = \bm{0}$. We first note that for any $\bm {v}_1 \in \bm {V_{EC}}(T)$, we have 
$\int_{f} \bm {v}_1 \, ds = 0$, for any $f \in \partial T$, since for all $f \in \partial T$, it holds that $\int_{f} b q \, ds = 0, \forall q \in P_1(f)$. This measn that $\int_{f} \bm {v}_0\, ds = 0$ for any $f \in \partial T$ and so, $\bm {v}_0 = 0$. Therefore, it is sufficient to show that $\bm{v}_1 = \bm{0}$. By the definition of $\bm{V_{CR}}(T)$, we can write $\bm {v}_1$ as follows for some constants $c_1, c_2, c_3$:
\begin{equation}
\bm{v}_1 = c_1\bm \psi_{ij} + c_2 \bm \psi_{jk} + c_3 \bm \psi_{ki}.
\end{equation}
The fact that $\int_{f} \bm{v}_1 \cdot \bm{n} \, q  ds = 0, \,\, \forall q \in P_1(f)$ leads to the following linear system: 
\begin{subeqnarray*} 
\int_{f_{ij}} (c_1\bm \psi_{ij}+c_2\bm \psi_{jk}+c_3 \bm \psi_{ki}) \cdot \bm n_{ij} \lambda_i  ds &=& 0, \\ 
\int_{f_{jk}}(c_1\bm \psi_{ij}+c_2\bm \psi_{jk}+c_3 \bm \psi_{ki})\cdot \bm n_{jk} \lambda_j  ds &=& 0, \\ 
\int_{f_{ki}} (c_1\bm \psi_{ij}+c_2\bm \psi_{jk}+c_3 \bm \psi_{ki}) \cdot \bm n_{ki} \lambda_k  ds &=& 0.
\end{subeqnarray*} 
Or equivalently, we have that 
\begin{equation}\label{system} 
\left ( \begin{array}{ccc} 
-2 & \bm{n}_{jk} \cdot \bm{n}_{ij} & \bm{n}_{ki} \cdot \bm{n}_{ij}  \\
 \bm{n}_{jk} \cdot \bm{n}_{ij} & -2 & \bm{n}_{ki} \cdot \bm{n}_{jk}   \\
 \bm{n}_{ij} \cdot \bm{n}_{ki} & \bm{n}_{jk} \cdot \bm{n}_{ki} & - 2 \end{array} \right ) \left ( \begin{array}{c} c_1 \\ c_2 \\ c_3 \end{array} \right ) = \left ( \begin{array}{c} 0 \\ 0 \\ 0 \end{array} \right ).  
 \end{equation} 
Since $|\bm n_{ij}\cdot \bm n_{ki}|+|\bm n_{jk}\cdot \bm n_{ki}| < 2$, the coefficient matrix in the equation \eqref{system} is diagonally dominant and so it is invertible, which implies $\bm {v}_1 = \bm 0$. This completes the proof. 
\end{proof}

\subsection{A new finite element space that satisfies the Korn's inequality}

In this subsection, we shall investigate the finite elements introduced in \cite{johnny2012family}. We observe that some of their finite element for the lowest order case do not satisfy the Korn's inequality and provide modification, which satisfies the Korn's inequality. For a given $T \in \mathcal{T}_h$, 
%
we let 
\begin{equation}
Q_f^{0}(T) = P_{0}(T)~
\hbox{for}~ d=2.
\end{equation}
and 
\begin{equation}
\bm{Q}_f^{0}(T) = (P_{0}(T))^3 \times \bm{n}_f~  \hbox{for}~ d=3,
\end{equation}
and define $Q^{0}(T) = \sum_{f \in \partial T} b_f {Q}_f^{0}(T)$ for $d = 2$ and $\bm Q^{0}(T)= \sum_{f \in \partial T} b_f \bm{Q}_f^{0}(T)$ for $d=3$. 

We now consider the lowest order case of finite elements introduced in \cite{johnny2012family}. We recall that the local space was given as follows:  
for $T \in \mathcal{T}_h$, 
\begin{equation}\label{guzman} 
\bm{V}_1(T) = \bm{V}_{1,T}^0 + \bm{curl} (b_T \bm{Y}),
\end{equation} 
where $\bm{V}^0_{1,T}$ is the well-known $H({\rm div})$-conforming finite element spaces, either ${\rm{BDM}}_1(T)$ or ${\rm{RT}}_1(T)$. The space $\bm{Y}$ is given as following:
\begin{equation}
\bm{Y} = \left \{ \begin{array}{ll} 
\bm{Y}_4 = Q^0(T) & \mbox{ for } d = 2; \\
\bm{Y}_5 = \bm Q^0(T)  & \mbox{ for } d = 3.
\end{array} \right.  
\end{equation}
From our sharp Korn's inequality, we can easily check if the above four finite elements given in \eqref{guzman} from different choices of $\bm{V}_{1,T}^0$ and $\bm{Y}$ satisfy the Korn's inequality immediately. This is presented in the following Table \ref{guzmantable}: 
\begin{table}[ht]
\begin{tabular}{ |p{2.17cm}||p{1.5cm}|p{0.5cm}|p{5.cm}|p{1cm}| }
\hline
\multicolumn{5}{|c|}{The six modified $H({\rm div})$ elements} \\
\hline\hline
{Elements}  & $\bm{V}_{1,T}^0$ & \bm{Y} & {\rm DOF} & Korn \\
\hline
$1^{\rm st}$ FEM (2D)  & ${\rm RT}_1(T)$     & $\bm{Y}_4$    & $ {\small{\begin{array}{l} 
\int_f \bm{v} \cdot \bm{n} q\, ds, \quad \forall q \in P_1(f), \\ \int_T \bm{v} \cdot \bm{q} \, dx, \quad \bm{q} \in (P_0(T))^2 \\ 
\int_f \bm{v} \cdot \bm{t} \, ds \end{array}}} $ & Yes \\
\hline 
$2^{\rm nd}$ FEM (2D) & ${\rm BDM}_1(T)$  & $\bm{Y}_4$   & $ {\small{\begin{array}{l} 
\int_f \bm{v} \cdot \bm{n} q\, ds, \quad \forall q \in P_1(f), \\  
\int_f \bm{v} \cdot \bm{t} \, ds \end{array}}} $  & Yes \\
\hline 
$3^{\rm rd}$ FEM (3D)  & ${\rm RT}_1(T)$      & $\bm{Y}_5$  &$ {\small{\begin{array}{l} 
\int_f \bm{v} \cdot \bm{n} q\, ds, \quad \forall q \in P_1(f), \\ \int_T \bm{v} \cdot \bm{q} \, dx, \quad \bm{q} \in (P_0(T))^3 \\ 
\int_f (\bm{v} \times \bm{n})\cdot\bm{r} \, ds, \quad \forall \bm{r} \in (P_0(f))^2 \end{array}}} $ & No \\
\hline 
$4^{\rm th}$ FEM (3D)  & ${\rm BDM}_1(T)$   & $\bm{Y}_5$   &$ {\small{\begin{array}{l} 
\int_f \bm{v} \cdot \bm{n} q\, ds, \quad \forall q \in P_1(f), \\ 
\int_f (\bm{v} \times \bm{n})\cdot\bm{r} \, ds, \quad \bm{r} \in (P_0(f))^2 \end{array}}} $  & No  \\
 \hline
\end{tabular}
\caption{Lowest FEMs introduced in \cite{johnny2012family}}\label{guzmantable} 
\end{table} 
We note that particularly, for the local space 
\begin{equation} 
\bm{V}_1(T) = \bm{V}_{1,T}^0 + \bm{curl} (b_T \bm{Y}_5)= \bm{V}_{1,T}^0 + \bm{curl} (b_T \bm{Q}^0(T)),
\end{equation} 
the degrees of freedom are not sufficient since the constant projected tangential component on the face being zero is not sufficient for the Korn's inequality. 
Therefore, we investigate how to remedy this situation. Namely, we shall modify the local space $\bm V_1(T)$ and the definition of degrees of freedom 
so that the Korn's inequality can be satisfied in the remainder of this subsection.  

We define the local space for the nonconforming finite element that satisfies the Korn's inequality :
\begin{equation}
\bm{V}^1(T) =  \bm{V}_{1,T}^0 +  {\bm{curl}} \left ( b_T \bm{Q}^*(T) \right ),
\end{equation}
where 
\begin{eqnarray*}
\bm{Q}^{*}(T) &=& \sum_{f \in \partial T} b_f \bm{Q}_f^*(T), 
\end{eqnarray*}
with 
\begin{eqnarray*}
\bm{Q}_f^*(T) = \left \{ \bm{q} \times \bm{n}_f : \bm{q} \in \bm{RM}(T),  \int_T (\bm{q} \times {\bm{n}}_f) \cdot (\bm{w} \times \bm{n}_f) b_T b_f \, dx  = 0, \,\, \bm{w} \in (P_0(T))^3 \right \}.  
\end{eqnarray*}
It is easy to see that the dimension of $\bm{Q}_f^*(T)$ is three.   
Therefore, the dimension of the space $\bm{Q}^*(T)$ is 12. Note that the space $ \bm{V}_{1,T}^0$ can be either $BDM_1$ or $RT_1$. But, we shall only consider $BDM_1$ for simplicity. Under this choice, it is well-known that a function $\bm{v} \in \bm{V}_{1,T}^0$ is uniquely determined by the following degrees of freedom: for all faces $f \in \partial T$,  
\begin{equation}
\langle \bm{v}\cdot \bm{n}_f, \bm{\mu} \rangle_f, \quad \forall \bm{\mu} \in ({P}_1(f))^3.  
\end{equation}
With this in mind, we can provide the following degrees of freedom that define a function $\bm{v} \in \bm{V}^1(T)$ uniquely: for all $f\in\partial T$,  
\begin{eqnarray}
\langle \bm{v}\cdot \bm{n}_f, \bm{\mu} \rangle_f, && \quad 
\forall \bm{\mu} \in ({P}_1(f))^3 \label{dof1}; \\
\langle \bm{v}\times \bm{n}_f, \bm{\kappa} \rangle_f, && \quad \forall \bm{\kappa} \in {\rm{RT}}_0(f). \label{dof2}
\end{eqnarray}
We can show that the following holds true:  
\begin{Lemma}
For any $T \in \mathcal{T}_h$, we have  
\begin{equation}
{\rm dim} \,{\bm{curl}} \left ( b_T \bm{Q}^*(T) \right ) = 4 \, {\rm dim} \, \bm{Q}_f^*(T).
\end{equation}
\end{Lemma}
\begin{proof}
The proof can be done similarly to the argument of Lemma 3.2 in \cite{johnny2012family}. Therefore, we complete the proof. 
\end{proof}
\begin{Theorem}
We have the following relation that 
\begin{equation}
\bm{V}^1(T) = \bm{V}_{1,T}^0 \oplus {\bm{curl}} \left ( b_T \bm{Q}^*(T) \right )
\end{equation}
and
\begin{equation}
{\rm dim} \bm{V}^1(T) = {\rm dim} \bm{V}_{1,T}^0 + {\rm dim} {\bm{curl}} \left ( b_T \bm{Q}^*(T) \right ). 
\end{equation}
Furthermore, any function $\bm{v} \in \bm{V}^1(T)$ is uniquely determined by the degrees of freedom \eqref{dof1} and \eqref{dof2}. 
\end{Theorem}
\begin{proof}
The proof can be done similarly to the argument of Theorem 3.3 in \cite{johnny2012family}. Therefore, we complete the proof. \end{proof}

\section{Conclusion}\label{con}

We have proven a Korn's inequality for the piecewise $H^1$ space. Our characterization of the  Korn's inequality considers to be the rotated version of what is characterized in \cite{mardal2006observation}.  We further have shown some sharpness of the Korn's inequality. This result has been applied to a number of finite elements to check that they satisfy the Korn's inequality or not. Also the minimal jump terms that appear in the Korn's inequality are explicit enough so that they can be used for construction of finite element spaces that satisfy the Korn's inequality.  

\bibliographystyle{plain}
\bibliography{bibKorn}

\end{document}